\documentclass[12pt]{amsart}
\usepackage[utf8x]{inputenc}
\usepackage{lastpage} 
\usepackage{amsmath} 
\usepackage{amsthm} 
\usepackage{amssymb} 
\usepackage{mathrsfs}
\usepackage[dvips,letterpaper,margin=1in]{geometry}
\usepackage{enumitem}
\usepackage{mathpazo}
\usepackage{color}

\setlist[enumerate]{itemsep=0mm,topsep=0mm}
\setlist[itemize]{itemsep=0mm,topsep=0mm}
\makeatletter
\newcommand{\MONTH}{%
  \ifcase\the\month
  \or January
  \or February
  \or March
  \or April
  \or May
  \or June
  \or July
  \or August
  \or September
  \or October
  \or November
  \or December
  \fi}
\renewcommand{\today}{\MONTH \ \number \day, \number \year}

\DeclareMathOperator{\ran}{ran}
\DeclareMathOperator{\ord}{ord}
\DeclareMathOperator{\supp}{supp}

\newcommand{\Span}{\mathrm{Span}}
\newcommand{\fs}{\mathrm{FS}}
\newcommand{\zfc}{\mathsf{ZFC}}

\newcommand{\F}{\mathbb{F}}

\newcommand{\B}{\mathbb{B}}
\newcommand{\fsm}{{\fs_{\mathrm{matrix}}}}

\newtheorem{theorem}{Theorem}
\newtheorem*{theorem*}{Theorem}
\newtheorem{lemma}[theorem]{Lemma}
\newtheorem*{lemma*}{Lemma}
\newtheorem{corollary}[theorem]{Corollary}
\newtheorem*{corollary*}{Corollary}

\newtheorem*{proposition*}{Proposition}
\newtheorem{proposition}[theorem]{Proposition}

\newcounter{pcount}
\theoremstyle{definition}
\newtheorem{definition}[theorem]{Definition}

\begin{document}

\title[Uncountably many colours and finite monochromatic sets]{Hindman-like theorems with uncountably many colours and finite monochromatic sets}

\author{David Fern\'andez-Bret\'on}
\address{Department of Mathematics\\
	University of Michigan\\
	2074 East Hall, 530 Church Street \\
	Ann Arbor, MI 48109-1043, U.S.A.\\
}
\curraddr{
Kurt G\"odel Research Center for Mathematical Logic \\
University of Vienna \\
W\"aringer Stra\ss e 25, 1090 Wien, Austria.
}

\email{david.fernandez-breton@univie.ac.at}
\urladdr{https://homepage.univie.ac.at/david.fernandez-breton/}

\author{Sung Hyup Lee}
\address{Department of Mathematics, University of California at Berkeley.}
\email{sunghlee@berkeley.edu}
\date{\today}



\begin{abstract}
A particular case of the Hindman--Galvin--Glazer theorem states that, for every partition of an infinite abelian group $G$ into two cells, there will be an infinite $X\subseteq G$ such that the set of its finite sums $\{x_1+\cdots+x_n \mid n\in\mathbb N\wedge x_1,\ldots,x_n\in X\text{ are distinct}\}$ is monochromatic. It is known that the same statement is false, in a very strong sense, if one attempts to obtain an uncountable (rather than just infinite) $X$. On the other hand, a recent result of Komj\'ath states that, for partitions into uncountably many cells, it is possible to obtain monochromatic sets of the form $\fs(X)$, for $X$ of some prescribed finite size, when working with sufficiently large Boolean groups. In this paper, we provide a generalization of Komj\'ath's result, and we show that, in a sense, this generalization is the strongest possible.
\end{abstract}

\maketitle

\setcounter{section}{0}
\section{Introduction}

Ramsey-type theorems are statements that assert the existence of rich monochromatic substructures whenever some big ambient structure is coloured (partitioned). A very notable such result, which constitutes the starting point for this paper, is Hindman's theorem~\cite{hindmanoriginal}, stated below (we actually state a very general version of this result, which is usually known as the {\em Hindman--Galvin--Glazer} theorem).

\begin{theorem*}[See~\cite{hindmanstrauss}]
Let $G$ be any infinite abelian group, and suppose that we colour its elements with finitely many colours. Then there is an infinite $X\subseteq G$ such that the set
 \begin{equation*}
  \fs(X)=\left\{\sum_{x\in F}x \mid F\subseteq X\text{ is finite nonempty}\right\}
 \end{equation*}
 of finite sums of elements of $X$ (with no repetitions) is monochromatic.
\end{theorem*}

In order to appropriately deal with generalizations to this theorem, it is convenient to introduce some notation. Given an abelian group $G$, the symbol $G\rightarrow(\kappa)_\theta^{\fs}$ will denote the statement that for every colouring of the elements of $G$ with $\theta$ colours, there exists an $X\subseteq G$ with $|X|=\kappa$ such that $\fs(X)$ is monochromatic; and $G\rightarrow[\kappa]_\theta^{\fs}$ is the exact same statement except that rather than requiring $\fs(X)$ to be monochromatic, we just require that it avoids at least one colour (thus the ``square-bracket'' statement is weaker than the ``round-bracket'' one, and this situation is reversed when one considers the negations of these statements). Hence, Hindman's theorem just asserts that $G\rightarrow(\omega)_k^{\fs}$ holds for every infinite abelian group $G$ and every finite number $k<\omega$. There are at least three questions that naturally arise upon encountering this result. The first one is whether it is possible to increase the number of colours to an infinite number, and still obtain infinite monochromatic sets. In Section~\ref{laseccioncorrespondiente} we answer this question in the negative. The second question that might naturally arise, is whether one can, still with finitely many colours, obtain uncountable monochromatic subsets. In~\cite{FBuncountable}, the first author of this paper answered that question in the negative, by proving that for every infinite abelian group $G$, it is the case that $G\nrightarrow(\omega_1)_2^{\fs}$. Later work of the first author and Assaf Rinot~\cite{FBAstrongfail} made it evident that, in fact, some much stronger negative statements are true. For example, it is the case that every uncountable abelian group $G$ satisfies $G\nrightarrow[\omega_1]_\omega^{\fs}$, that is, there is a colouring of the group $G$ with $\omega$ colours such that for every uncountable $X\subseteq G$, not only does $\fs(X)$ fail to be monochromatic, but in fact it contains elements of all possible colours (we say that $\fs(X)$ is panchromatic). The question of whether the number of colours in this result can be increased from $\omega$ to $\omega_1$ turns out to be independent from the $\zfc$ axioms, with both alternatives $(\forall G)(G\nrightarrow[\omega_1]_{\omega_1}^\fs)$ and $\mathbb R\rightarrow[\omega_1]_{\omega_1}^\fs$ being consistent (the latter depending on a very mild large cardinal assumption). Furthermore, results from~\cite{FBAstrongfail} entail that for many cardinals $\kappa$ (for example, if $\kappa$ is the successor of a regular cardinal) it is in fact the case that every abelian group $G$ of cardinality $\kappa$ satisfies $G\nrightarrow[\kappa]_\kappa^{\fs}$; in fact, it is consistent to have $\mathbb R\nrightarrow[\mathfrak c]_{\mathfrak c}^{\fs}$, where $\mathbb R$ is the additive group of real numbers.

Surprisingly, there is a third question that arises naturally from consideration of Hindman's theorem, by increasing the number of colours and \textit{simultaneously} decreasing the size of the expected monochromatic set. That is, we can also consider the question of whether one can obtain an analog of Hindman's theorem, with infinitely many colours, where the monochromatic sets that we request are finite. Komj\'ath~\cite{komjath} initiated this line of research by proving that, for every finite $n$ and infinite $\kappa$, there exists a (sufficiently large) $\lambda$ such that, if $\mathbb B(\lambda)$ is the (unique up to isomorphism) Boolean group of cardinality $\lambda$, then $\mathbb B(\lambda)\rightarrow(n)_\kappa^{\fs}$. In fact, even more is true. Komj\'ath proved that for every $\kappa$, there is a (sufficiently large) $\lambda$ such that, whenever we colour the Boolean group $\mathbb B(\lambda)$ with $\kappa$ many colours, there is a $\kappa\times n$ matrix of elements of $\mathbb B(\lambda)$, $(x_{\alpha,i} \mid \alpha<\kappa\wedge i<n)$, such that the set
\begin{equation*}
\fsm(x_{\alpha,i} \mid \alpha<\kappa\wedge i<n)=\left\{x_{\alpha_1,i_1}+\cdots+x_{\alpha_k,i_k} \mid i_1<\cdots<i_k<n\wedge\alpha_1,\ldots,\alpha_k<\kappa\right\}
\end{equation*}
is monochromatic. From now on, we will abbreviate this statement by means of the symbol $\mathbb B(\lambda)\rightarrow(\kappa\times n)_\kappa^\fsm$ (and the analogous symbol when $\mathbb B(\lambda)$ is replaced by some other group). Note that, since $\fsm(x_{\alpha,i} \mid \alpha<\kappa\wedge i<n)$ contains $\fs(x_{\alpha,i} \mid i<n)$ for every fixed $\alpha$, the statement $G\rightarrow(\kappa\times n)_\kappa^\fsm$ is always stronger than the statement $G\rightarrow(n)_\kappa^\fs$. Along similar lines, Carlucci~\cite{carluccihindmantype} has proved results where Boolean groups are coloured with uncountably many colours, and uncountable sets $X$ are obtained where a carefully restricted subset of $\fs(X)$ (containing only those finite sums whose number of summands belongs to a set of some specific form) is monochromatic.

In this paper we generalize Komj\'ath's result above for $n=2$ by showing that, for every $\kappa$, there is a $\lambda$ such that every abelian group $G$ of cardinality $\lambda$ will satisfy $G\rightarrow(\kappa\times 2)_\kappa^\fsm$. Our $\lambda$ is just the cardinal $\left(2^\kappa\right)^+$, which is significantly smaller than the one used by Komj\'ath in his result. Moreover, we also show that this $\lambda$ is optimal, in the sense that $\B(2^\kappa)\nrightarrow(2)_\kappa^\fs$ (and \textit{a fortiori}, $\B(2^\kappa)\nrightarrow(\kappa\times 2)_\kappa^\fsm$). Extremely surprisingly, however, our attempt to further generalize these results for $n=3$ yields a negative result: we also show that there are arbitrarily large groups $G$ such that $G\nrightarrow(3)_\omega^\fs$. So Komj\'ath's result for monochromatic $\fs$-sets generated by three or more elements cannot be extended to all abelian groups of a given size.

\noindent{\bf Notation:} Most of the standard notational conventions of set theory are followed here. For example, every ordinal $\alpha$ is a von Neumann ordinal (and so, set-theoretically, $\alpha$ is just the set of all ordinals $\xi<\alpha$), and each cardinal is identified with the least ordinal of that cardinality. If $\kappa$ is an infinite cardinal, then $\kappa^+$ is its successor cardinal. Given a cardinal $\kappa$, we recursively define the beth sequence that starts at $\kappa$ by letting $\beth_0(\kappa) = \kappa$, $\beth_{\alpha + 1} = 2^{\beth_\alpha(\kappa)}$, and $\beth_\alpha(\kappa) = \sup\{\beth_\beta(\kappa) \mid  \beta < \alpha\}$ for limit $\alpha$. We use exponential notation for both the cardinal exponential function, and the set-theoretic exponential operation (taking a set of functions). Thus, an expression such as $\kappa^\lambda$ might denote either the set $\{f \mid f:\lambda\longrightarrow\kappa\}$, or the cardinality of that set. We are confident that the context will always be sufficient to determine which of the two meanings of $\kappa^\lambda$ should be inferred every time the symbol occurs, and that therefore there will be no confusion arising from this.
   
Given an infinite cardinal $\lambda$, we use the symbol $\B(\lambda)$ to denote the unique (up to isomorphism) Boolean group of cardinality $\lambda$ (recall that a Boolean group is one in which every element has order $2$), whose most friendly incarnation is $([\lambda]^{<\omega}, \vartriangle)$

Something must be said also about how the injectivity of divisible groups provides a structural theorem that yields much information about how abelian groups look like. The result (see e.g. \cite[p. 123]{summable-sparse} for a detailed explanation) is that, if $G$ is an arbitrary abelian group, then we can embed $G$ into a direct sum of the form $\bigoplus_{\alpha < \lambda} G_\alpha$, where each $G_\alpha$ is either equal to $\mathbb Q$, or to a Pr\"ufer group $\mathbb Z[p^\infty]$, for $p$ a prime number. Thus for every abelian group $G$, there is a pairwise disjoint sequence of indices $(I_p \mid p\in P)$, where $\mathbb P$ is the union of $\{0\}$ with the collection of prime numbers, such that $G$ embeds in $\bigoplus_{p\in\mathbb P}\left(\bigoplus_{\alpha\in I_p}\mathbb Z[p^\infty]\right)$, where we stipulate that $\mathbb Z[0^\infty]=\mathbb Q$ for notational convenience. In this context, for $p\in\mathbb P$ we will denote by $\pi_p$ the projection onto the indices from $I_p$, in other words, $\pi_p(x)=x\upharpoonright I_p\bigoplus_{\alpha\in I_p}\mathbb Z[p^\infty]$. A group that has special importance for us is the one-dimensional torus $\mathbb T=\mathbb R/\mathbb Z$. Typically we will identify an element of $\mathbb T$, which formally is the coset of a real number modulo $\mathbb Z$, with its unique representative belonging to $[0,1)$. With this identification, we can describe $\mathbb Z[p^\infty]=\left\{\frac{a}{p^k}\in\mathbb Q/\mathbb Z\subseteq\mathbb T \mid a\in[0,1)\wedge k\in\mathbb N\right\}$, for $p$ a prime number. Sometimes it is also convenient to think of $\mathbb Q$ as embedded in $\mathbb T$ (which can be easily done, e.g. via the mapping $q\longmapsto\sqrt{2} q+\mathbb Z$). Hence on occasion we will think of the groups $G_\alpha$ as countable subgroups of $\mathbb T$. Finally, for an element $x$ of any direct sum of groups, we will let $\supp(x)$ denote the (finite) set of indices $\beta$ such that $x(\beta)\neq 0$, and we will let $\sigma(x)$ denote the (finite) sequence of non-zero entries of $x$ (essentially, $\sigma(x)$ is just $x\upharpoonright\supp(x)$; formally it is the result of composing the inverse of the Mostowski collapse of $\supp(x)$ with the sequence $x\upharpoonright\supp(x)$).

\ 

\noindent{\bf Organization of this paper:} This paper contains 4 sections in addition to this Introduction. In Section 2, we generalize Komj\'ath's result, and improve his upper bound, for $n=2$. In Section 3 we proceed to prove that this same result cannot be generalized any further for $n\geq 3$. In Section 4, we show that the upper bound obtained in Section 2 is optimal. Finally, in Section 5 we prove some miscellaneous negative results along the same lines as the other results in the paper.

\section{Generalization of Komj\'ath's results for $n=2$}\label{sectkomjath2}

In this section, we will proceed to prove that for every $\kappa$ there is a $\lambda$ such that every abelian group $G$ of cardinality $\lambda$ satisfies $G\rightarrow(2)_\kappa^\fsm$.

\begin{definition}
If $G$ is a group and $(g_\alpha \mid \alpha<\eta)$ is a transfinite sequence of elements of $G$, we will say that the sequence is \textbf{independent} if for every $\alpha<\eta$ the element $g_\alpha$ does not belong to the subgroup of $G$ generated by $\{g_\xi \mid \xi<\alpha\}$.
\end{definition}

Note that any independent sequence $\langle g_\alpha \mid \alpha<\eta\rangle$ will satisfy that, whenever $\alpha_i<\beta_i$ ($i<2$) are such that $g_{\alpha_0} - g_{\beta_0} = g_{\alpha_1} - g_{\beta_1}$, then $\alpha_0 = \alpha_1$ and $\beta_0 = \beta_1$. We start with an easy lemma in this direction.

\begin{lemma}\label{lemma:indseq}
Let $G$ be an abelian group, and let $X\subseteq G$ be a subset of uncountable cardinality $\kappa$. Then there exists an independent sequence of length $\kappa$ whose range is contained in $X$.
\end{lemma}

\begin{proof}
Recursively pick $g_\alpha\in X$ that does not belong to the subgroup generated by $\{g_\xi \mid \xi<\alpha\}$. This can always be done, as long as $\alpha<\kappa$, because the latter subgroup has cardinality at most $\max\{|\alpha|,\omega\}<\kappa$.
\end{proof}

The previous lemma is false if we let $G$ be a countable group. For example, in $\mathbb Z$ it is easily checked that there are no infinite independent sequences (if $(x_n \mid n<\omega)$ were an infinite independent sequence, then the sequence $I_0\subseteq I_1\subseteq\cdots\subseteq I_n\subseteq\cdots$, where each $I_n$ is the subgroup generated by $\{x_i \mid i\leq n\}$, would be a strictly increasing sequence of ideals in the Noetherian ring $\mathbb Z$).

The following theorem generalizes Komj\'ath's result~\cite[Theorem 2]{komjath}, in the specific case that $n=2$.

\begin{theorem}\label{theorem:2hom}
   Let $\kappa$ be an infinite cardinal, and let $\lambda=(2^\kappa)^+$. Then for every abelian group $G$ of cardinality $\lambda$, we have that $G \rightarrow (\kappa\times 2)^\fsm_\kappa$.
\end{theorem}

\begin{proof}
Suppose that $G$ is an abelian group with $ \mid G \mid =\lambda$, and choose an independent sequence $(g_\alpha \mid \alpha<\lambda)$ in $G$. Define $d:[\lambda]^2\longrightarrow\kappa$ by $d(\{\alpha<\beta\})=c(g_\beta-g_\alpha)$. By the Erd\H{o}s--Rado theorem, $\lambda\rightarrow(\kappa^+)_\kappa^2$ (in fact, for what we are about to do it suffices that $\lambda\rightarrow(\kappa+\kappa)_\kappa^2$), thus we can choose two increasing sequences of ordinals $(\alpha_\xi \mid \xi<\kappa)$ and $(\gamma_\eta \mid \eta<\kappa)$, as well as an ordinal $\beta$, satisfying
\begin{equation*}
\sup_{\xi<\kappa}\alpha_\xi<\beta<\gamma_0
\end{equation*}
and a colour $\delta$ such that $d``[\{\alpha_\xi \mid \xi<\kappa\}\cup\{\beta\}\cup\{\gamma_\eta \mid \eta<\kappa\}]^2=\{\delta\}$. For $\xi<\kappa$, define $x_{\xi,0}=g_\beta-g_{\alpha_\xi}$ and $x_{\xi,1}=g_{\gamma_\xi}-g_\beta$. Since the sequence $(g_\alpha \mid \alpha<\lambda)$ is independent, the entries of our matrix are pairwise distinct. Furthermore, for each choice of $\xi,\eta<\kappa$, we have that
\begin{eqnarray*}
c(x_{\xi,0}) & = & c(g_\beta-g_{\alpha_\xi})=d(\{\alpha_\xi,\beta\})=\delta, \\
c(x_{\xi,1}) & = & c(g_{\gamma_\eta}-g_\beta)=d(\{\beta,\gamma_\eta\})=\delta, \\
c(x_{\xi,0}+x_{\eta,1}) & = & c(g_\beta-g_{\alpha_\xi}+g_{\gamma_\eta}-g_\beta)=c(g_{\gamma_\eta}-g_{\alpha_\xi})=d(\{\alpha_\xi,\gamma_\eta\})=\delta,
\end{eqnarray*}
and this finishes the proof.
\end{proof}

Komj\'ath's result~\cite[Theorem 2]{komjath} establishes that, given $n$, letting $\lambda=\beth_{2^{n-1}}(\beth_{2^{n-1}-1}(\kappa)^+)^+$ yields that $\mathbb B(\lambda)\rightarrow(\kappa\times n)_\kappa^\fsm$. In particular, for $n=2$ he needs to consider groups of cardinality $\beth_1(\kappa^+)^+=(2^{\kappa^+})^+$. In Theorem~\ref{theorem:2hom} above, we were able to simultaneously lower this size to $(2^\kappa)^+$ and at the same time improve the result to arbitrary abelian groups, although only in the case $n=2$. In Section 4 we will show that the $(2^\kappa)^+$ in this result is optimal, in the sense that there are groups $G$ of cardinality $2^\kappa$ such that $G\nrightarrow(2)_\kappa^\fs$. In Section 3, we will show that the number $n=2$ is also optimal, in the sense that there are groups $G$ of arbitrarily large size such that $G\nrightarrow(3)_\omega^\fs$.

\section{Komj\'ath's result cannot be generalized for $n\geq 3$}

Given Komj\'ath's results~\cite{komjath}, as well as our Theorem~\ref{theorem:2hom} above, it would be desirable to generalize these results to $n\geq 3$. That is, we would like to obtain, given a $\kappa$ and an $n\geq 3$, a sufficiently large cardinal $\lambda$ such that all abelian groups $G$ of cardinality $\lambda$ satisfy $G\rightarrow(\kappa\times 3)_\kappa^\fsm$, or at least $G\rightarrow(3)_\kappa^\fs$. In this section we will see, however, that this is impossible. The first step towards the proof of such impossibility result is the following Lemma, whose proof was communicated to the first author by Imre Leader and independently also by Julian Sahasrabudhe in personal communications, and is reproduced here with their kind permission.

\begin{lemma}\label{leader}
There are no three distinct vectors $\vec x,\vec y,\vec z\in\mathbb R^n$ such that $\| \vec x\| = \| \vec y \| =\| \vec z \| =\| \vec x+\vec y\| =\| \vec x+\vec z \| =\| \vec y+\vec z\| =\| \vec x+\vec y+\vec z\| $.
\end{lemma}

\begin{proof}
Suppose, on the contrary, that $\vec x,\vec y,\vec z$ are three distinct vectors such that $r=\langle\vec x,\vec x\rangle=\langle\vec y,\vec y\rangle=\langle\vec z,\vec z\rangle=\langle\vec x+\vec y,\vec x+\vec y\rangle=\langle\vec y+\vec z,\vec y+\vec z\rangle=\langle\vec x+\vec z,\vec x+\vec z\rangle=\langle\vec x+\vec y+\vec z,\vec x+\vec y+\vec z\rangle$ (in particular, $r\neq 0$). Notice that
\begin{equation*}
r=\langle\vec x+\vec y,\vec x+\vec y\rangle=\langle\vec x,\vec x\rangle+2\langle\vec x,\vec y\rangle+\langle\vec y,\vec y\rangle=2\langle\vec x,\vec y\rangle+2r,
\end{equation*}
which implies that $\langle\vec x,\vec y\rangle=-\frac{1}{2}r$. A similar argument shows that $\langle\vec y,\vec z\rangle=\langle\vec x,\vec z\rangle=-\frac{1}{2}r$. Now
\begin{eqnarray*}
r & = & \langle\vec x+\vec y+\vec z,\vec x+\vec y+\vec z\rangle=\langle\vec x,\vec x\rangle+\langle\vec y,\vec y\rangle+\langle\vec z,\vec z\rangle+2\langle\vec x,\vec y\rangle+2\langle\vec y,\vec z\rangle+2\langle\vec x,\vec z\rangle \\
 & = & 3r+2\left(-\frac{3}{2}r\right)=0,
\end{eqnarray*}
which is a contradiction.
\end{proof}

Now we can explicitly state the theorem about the impossibility of obtaining further generalizations of Komj\'ath's result.

\begin{theorem}\label{impossiblefor3}
For each infinite cardinal $\lambda$, there is an abelian group $G$ of cardinality $\lambda$ such that $G\nrightarrow(3)_\omega^\fs$.
\end{theorem}

\begin{proof}
Given $\lambda$, let $G$ be either $\bigoplus_{\alpha<\lambda}\mathbb Z$ (the free abelian group on $\lambda$ generators), or $\bigoplus_{\alpha<\lambda}\mathbb Q$. Each element $x\in G$ is a sequence of integers (or of rational numbers), almost all of which equal zero. Hence it makes sense to define $c(x)=\sum_{\alpha<\lambda}x(\alpha)^2$. This gives us a colouring $c:G\longrightarrow\mathbb Z$ on $\omega$ many colours. Suppose that there were three distinct elements $x,y,z$ such that $c(x)=c(y)=c(z)=c(x+y)=c(x+z)=c(y+z)=c(x+y+z)$. Let $I=\supp(x)\cup\supp(y)\cup\supp(z)$. By replacing $x,y,z$ with $x\upharpoonright I,y\upharpoonright I,z\upharpoonright I$ (that is, ignoring all entries not in $I$, which equal $0$ anyway), we can think of $x,y,z$ as vectors in $\mathbb Z^n\subseteq\mathbb R^n$ (or in $\mathbb Q^n\subseteq\mathbb R^n$), where $n= | I |$. Then we will have that $\| x \| =\| y \| =\|z\| =\| x+y\| = \| x+z\| = \|y+z\| =\|x+y+z\| $, contradicting Lemma~\ref{leader}.
\end{proof}

Suppose that we are given $\kappa$, $n$, and and an abelian group $G$. Let $\lambda=\beth_{2^{n-1}-1}(\kappa)^+$, as in Komj\'ath's~\cite[Theorem 1]{komjath} for $\kappa$ and $n$. Then there are two cases where we can arrive to conclusions:
\begin{enumerate}
\item If $G$ has at least $\lambda$ elements of order $2$, then we know that $G\rightarrow(n)_\kappa^\fs$ (since $\mathbb B(\lambda)$ embeds in $G$, and hence every colouring of $G$ induces a corresponding colouring in $\mathbb B(\lambda)$).
\item On the other hand, if $G$ has $\kappa$ or fewer elements of finite order, then we know that $G\nrightarrow(n)_\kappa^\fs$. This is because, by the injectivity of divisible groups (as explained in the Introduction), we can think of $G$ as a subgroup of $\bigoplus_{p\in\mathbb P}\left(\bigoplus_{\alpha\in I_p}\mathbb Z[p^\infty]\right)$. Now we define a colouring $c$ of $G$ as follows. Notice that elements of finite order are those $x\in G$ for which $\pi_0(x)=0$, so we just colour each of these with a different colour (we do have enough colours, because there are less than $\kappa$ many elements of finite order). Now colour every $x\in G$ of infinite order with colour $\left(\sum_{\alpha\in I_0}x(\alpha)^2,y\right)$, where $y$ is the unique element satisfying $\pi_p(y)=\pi_p(x)$ for every $p\in\mathbb P\setminus\{0\}$, and $\pi_0(y)=0$. The same argument as in the proof of Theorem~\ref{impossiblefor3} shows that $c$ cannot carry monochromatic sets of the form $\fs(x,y,z)$.
\end{enumerate}

Now, if neither of these two possibilities hold, then we do not have yet a way of deciding whether or not $G\rightarrow(n)_\kappa^\fs$. In order to provide some partial answer to this question, we introduce the following definition.

\begin{definition}\label{defpattern}
Let $n<\omega$, and let $m\geq2$. A sequence $(x_k \mid k<n)$ will be called an \textbf{n-adequate pattern modulo $m$} if, for some $l<\omega$, we have $(\forall k<n)(x_k\in(\mathbb Z/m\mathbb Z)^l)$, and $\sigma``[\fs(x_k \mid k<n)]$ is a singleton, where $\sigma(x)$ denotes the sequence of non-zero entries of $x\in(\mathbb Z/m\mathbb Z)^l$.
\end{definition}

The construction of an independent family by Komj\'ath in~\cite[Lemma 1]{komjath} provides us with $n$-adequate patterns modulo $2$, for every $n<\omega$ (his construction yields an $n$-adequate pattern living in $(\mathbb Z/2\mathbb Z)^{2^n}$). For any arbitrary $m$, the sequence \newline $((1,-1,0),(0,1,-1))$ in $(\mathbb Z/m\mathbb Z)^3$ provides us with a $2$-adequate pattern modulo any arbitrary $m$. On the other hand, Theorem~\ref{impossiblefor3} tells us that it is hopeless to look for $n$-adequate patterns ``modulo 0'' (i.e. in some $\mathbb Z^l$ rather than $(\mathbb Z/m\mathbb Z)^l$), for $n\geq 3$. The issue of whether an arbitrary group $G$ satisfies $G\rightarrow(n)_\kappa^\fs$ for various $n$ and $\kappa$, as per the discussion immediately before Definition~\ref{defpattern}, is essentially equivalent to the question of whether there are $n$-adequate patterns modulo $m$ for various $n,m\in\mathbb N$, so that the mapping $\sigma$ constitutes, in a sense, the ``canonical'' colouring for abelian groups (in the sense that finding monochromatic subsets for an arbitrary colouring is equivalent to finding monochromatic subsets for this particular colouring, i.e. patterns). The next two theorems make this precise (thus providing a ``canonization theorem'' of sorts).

\begin{theorem}
Suppose that there exists an $n$-adequate pattern modulo $m$. Then for every $\kappa$ there exists a $\lambda$ such that every abelian group $G$ containing at least $\lambda$ elements of order $m$ will satisfy $G\rightarrow(n)_\kappa^\fs$.
\end{theorem}

\begin{proof}
Let $(x_1,\ldots,x_n)$ be an $n$-adequate pattern modulo $m$. Suppose that the $x_i$ belong to $\bigoplus_{i<L}(\mathbb Z/m\mathbb Z)$, and suppose that $\sigma(x_i)=(k_1,\ldots,k_l)$ has length $l$. Let $\lambda=\beth_{l-1}(\kappa)^+$, so that by Erd\H{o}s--Rado we have $\lambda\rightarrow(L)_\kappa^l$. Take an arbitrary abelian group $G$ with $\lambda$ elements of order $m$, and suppose that $c:G\longrightarrow\kappa$. Using Lemma~\ref{lemma:indseq} (with $X$ the set of elements of order $m$), pick an independent sequence $(g_\xi \mid \xi<\lambda)$ consisting of elements of $G$ of order $m$. Define a colouring $d:[\lambda]^l\longrightarrow\kappa$ by stipulating that $d(\{\xi_1,\ldots,\xi_l\})=c(k_1g_{\xi_1}+\cdots+k_lg_{\xi_l})$, whenever $\xi_1<\cdots<\xi_l<\lambda$. By our choice of $\lambda$, there will be ordinals $\beta_1,\ldots,\beta_L$ such that $d``[\{\beta_1,\ldots,\beta_L\}]^l=\{\delta\}$. Now define $y_i=\sum_{j<L}x_i(j) g_{\beta_j}$ for $i\in\{1,\ldots,n\}$. We claim that, for every $y\in\fs(\{y_1, \dots, y_n\})$, it is the case that $c(y)=\delta$. Since each $g_\beta$ has order $m$, we have that $kg_\beta=k'g_\beta$ if $k\equiv k'\mod m$, and in particular if $m\mid k$ then $kg_\beta=0$. So if $y=y_{i_1}+\cdots+y_{i_t}$, letting $x=x_{i_1}+\cdots+x_{i_t}$ we get that 
\begin{eqnarray*}
y & = & \sum_{j<L}x_{i_1}(j)g_{\beta_j}+\cdots+\sum_{j<L}x_{i_t}(j)g_{\beta_j}=\sum_{j<L}(x_{i_1}(j)+\cdots+x_{i_t}(j))g_{\beta_j} \\ 
 & = & \sum_{j<L}x(j)g_{\beta_j}=\sum_{j\in\supp(x)}x(j)g_{\beta_j}=k_1g_{\xi_1}+\cdots+k_lg_{\xi_l},
\end{eqnarray*}
where $\supp(x)=\{\xi_1,\ldots,\xi_l\}\in[\{\beta_1,\ldots,\beta_L\}]^l$ (this works because $\sigma(x)=(k_1,\ldots,k_l)$ by assumption). Hence 
\begin{equation*}
c(y)=c(k_1g_{\xi_1}+\cdots+k_lg_{\xi_l})=d(\{\xi_1,\ldots,\xi_l\})=\delta,
\end{equation*}
and we are done.
\end{proof}

\begin{theorem}
Suppose that $n,m\in\mathbb N$ have the property that there is a $\lambda$ such that every abelian group $G$ with at least $\lambda$ elements of order $m$ satisfies $G\rightarrow(n)_\omega^\fs$. Then there is an $n$-adequate pattern modulo $m$.
\end{theorem}

\begin{proof}
Consider $G=\bigoplus_{\alpha<\lambda}(\mathbb Z/m\mathbb Z)$. The mapping $\sigma:G\rightarrow(\mathbb Z/m\mathbb Z)^{<\omega}$ that sends each $x$ to its sequence of non-zero entries colours $G$ with countably many colours. Clearly a sequence $(x_i \mid i<n)$ with $\fs(x_i \mid i<n)$ monochromatic will immediately yield an $n$-adequate pattern modulo $m$.
\end{proof}

Thus, a complete description of the groups $G$ for which $G\rightarrow(n)_\kappa^\fs$ is equivalent to an answer to the question of the existence of $n$-adequate patterns modulo $m$ for various $m$. The Chinese Remainder Theorem suggests that we might only need to search for these patterns if $m$ is the power of a prime number. The lemma below shows that, in fact, it suffices to consider the situation where $m$ is a prime number.

\begin{lemma}\label{lemma:lambdaembed}
   Let $\lambda$ be an uncountable regular cardinal, and let $G$ be an abelian group with $\lambda$ many elements of finite order. Then there is a prime number $p$, and a subgroup $H\subseteq G$, with $ |H| =\lambda$, such that all elements of $H$ have order $p$.
\end{lemma}

\begin{proof}
Using again the injectivity of divisible groups, as described in the Introduction, we know we can embed $G$ into $\bigoplus_{\alpha < \lambda} G_\alpha$ where each $G_\alpha$ is a countable subgroup of $\mathbb T=\mathbb R/\mathbb Z$. Considering all elements of order $m$ of $G$, by the $\Delta$-system lemma we can pick $\lambda$ many of them, say $\{x_\alpha \mid \alpha<\lambda\}$, whose supports form a $\Delta$-system (this utilizes the fact that for every given support, there are at most countably many elements of $G$ with that given support, which follows from the fact that each $G_\alpha$ is countable), let us denote by $r$ the root of this $\Delta$-system. By the pigeonhole principle, we may assume that all of the $x_\alpha\upharpoonright r$ are equal. Since each $x_\alpha$ has order $m$, then for each $\xi\in r$ we have that $mx_\alpha(\xi)=0$. Thus if we partition $\lambda$ into $\lambda$ many pairwise disjoint subsets of size $m$, $\{F_\alpha \mid \alpha<\lambda\}$, and we define $y_\alpha=\sum_{\beta\in F_\alpha}x_\beta$, we will have that $y_\alpha\upharpoonright r$ is identically zero, $y_\alpha$ will still have order $m$, and the supports of the $y_\alpha$ will be pairwise disjoint. If $p$ is a prime number dividing $m$, with $m=pk$, then it is not hard to check now that the sequence $(k y_\alpha \mid \alpha<\lambda)$ consists of $\lambda$ many distinct elements, all of which have order $p$.
\end{proof}

We close this section with a couple of final notes. As we pointed out above, $n$-adequate patterns modulo $2$ are known to exist for all $n$, and so are $2$-adequate patterns modulo $m$ for all natural numbers $m$. After almost a year of not knowing anything else, the first author was recently made aware (via a personal communication from Andy Zucker) that Christopher Cox recently found, aided by a computer, an example of a $3$-adequate pattern modulo $3$. Whether or not there are $n$-adequate patterns modulo $m$ for $n\neq 2\neq m$ and $(n,m)\neq(3,3)$ is still wide open (as mentioned above, in order to completely characterize abelian groups satisfying finitary Hindman-like theorems, it suffices to determine the existence of $n$-adequate patterns modulo prime numbers $p$).

The idea of using the mapping $\sigma$ (mapping an element of a direct sum to the sequence of its nonzero entries) as a canonical colouring, in order to prove algebraic Ramsey-theoretic results on large abelian groups, was also recently used by Leader and Russell~\cite{leader-russell} to prove that every abelian group $G$ of cardinality at least $\beth_\omega$ without elements of order $4$ has the property that for every finite colouring $c:G\longrightarrow k$, there exists an infinite $X\subseteq G$ such that $X+X$ is monochromatic\footnote{Leader and Russell's result is not stated in as general a form, but their proof easily yields the more general version stated here.}. This is also one of the central ideas in the recent proof of Komj\'ath, Leader, Russel, Shelah, D. Soukup and Vidny\'anszky~\cite{6-authors} that it is consistent, modulo some large cardinals, that the real numbers have the same property (namely, for every $c:\mathbb R\longrightarrow k<\omega$ there exists an infinite $X\subseteq\mathbb R$ such that $X+X$ is monochromatic).

\section{A lower bound}

In this section, we will show that the upper bounds from Section~\ref{sectkomjath2} are optimal. Recall that, in Theorem~\ref{theorem:2hom}, it was established that for every infinite $\kappa$, every abelian group $G$ of cardinality $(2^\kappa)^+$ satisfies $G\rightarrow(2)_\kappa^\fs$. The theorem below establishes that the number $(2^\kappa)^+$ in this result is the best possible.

\begin{theorem}\label{lowerbound}
   For every infinite $\kappa$, $\B(2^\kappa) \not\rightarrow (2)_\kappa^\fs$.
\end{theorem}
\begin{proof}
Let $\kappa$ be an infinite cardinal. We think of the group $\B(2^\kappa)$ as consisting of the set $[2^\kappa]^{<\omega}$ of finite subsets of $2^\kappa$, equipped with symmetric difference $\bigtriangleup$ as its group operation. We also think of $2^\kappa$ as the set of branches through the full binary tree of height $\kappa$. Hence $2^\kappa$ is equipped with a natural linear order, namely the lexicographic one, which we will denote simply by $<$. Given two distinct elements $f,g\in 2^\kappa$, we will denote by $\Delta(f,g)=\min\{\alpha<\kappa \mid f(\alpha)\neq g(\alpha)\}$. Hence $f<g$ iff $f(\Delta(f,g))<g(\Delta(f,g))$. As a useful convention, we will stipulate that $\Delta(f,f)=\kappa$.

We define a colouring $c$ of $\B(2^{\kappa})=[2^{\kappa}]^{<\omega}$ as follows: Given an $x\in[2^{\kappa}]^{<\omega}$, suppose that $x=\{x_0,\ldots,x_{n-1}\}$, where $x_0<x_1<\cdots<x_{n-1}$. We let $c(x)$ be the double-indexed finite sequence $(\Delta(x_i,x_j) \mid i,j<n)$. Thus we can think of $c(x)$ as a finite (square) symmetric matrix whose diagonal entries all equal $\kappa$, and with ordinals less than $\kappa$ off the diagonal. Clearly there are $\left | [\kappa]^{<\omega}\right | =\kappa$ many such matrices, thus $c$ colours $\B(2^{\kappa})$ with $\kappa$ many colours. We will now argue that it is not possible to have $x,y\in[2^\kappa]^{<\omega}$ such that $c(x)=c(y)=c(x \bigtriangleup y)$.

So suppose that $x,y\in[2^\kappa]^{<\omega}$ are two elements satisfying $c(x)=c(y)=c(x\bigtriangleup y)$, and furthermore suppose that these elements were chosen with $n= | x | = | y | = | x\bigtriangleup y |$ smallest possible. Let $\alpha=\min(\ran(c(x)))$, the least ordinal that occurs as an entry of the matrix $c(x)$. We let $x=\{x_0,\ldots,x_{n-1}\}$, $y=\{y_0,\ldots,y_{n-1}\}$ and $x\bigtriangleup y=\{z_0,\ldots,z_{n-1}\}$, all three ordered lexicographically. Suppose that $i<j$ are such that $\alpha$ is the $(i,j)$-th entry of $c(x)$, that is, $\alpha=\Delta(x_i,x_j)$. Let $f=x_i\upharpoonright\alpha$. Since $\alpha$ is the smallest of all the $\Delta(x_i,x_k)$, it follows that $x_k\upharpoonright\alpha=f$ for all $k<n$. Similarly, $y_k\upharpoonright\alpha=f$ for all $k<n$, and of course $z_k\upharpoonright\alpha=f$ for all $k<n$. We will also have that $\Delta(y_i,y_j)=\alpha$; and moreover will have that $x_i(\alpha)=0$, and $y_i(\alpha)=0$. Now for each $k\neq i$, by looking at the element $\Delta(x_i,x_k)$ we can know whether $x_k(\alpha)$ equals $0$ or $1$ (the former if $\Delta(x_i,x_k)>\alpha$, the latter otherwise). Since $\Delta(x_i,x_k)=\Delta(y_i,y_k)=\Delta(z_i,z_k)$, we will have that $\{k<n \mid x_k(\alpha)=0\}=\{k<n \mid y_k(\alpha)=0\}=\{k<n \mid z_k(\alpha)=0\}$, let us call this set $I$. Letting $x'=\{x_k \mid k\in I\}$, $y'=\{y_k \mid k\in I\}$ and $z=\{z_k \mid k\in I\}$, we can see now that $c(x')=c(y')=c(z)$ (since each of these matrices arises from specifically taking the $(k,k')$-th entries of the matrix $c(x),c(y),c(z)$, for $k,k'\in I$). Moreover, $x'\subseteq x\setminus\{x_j\}$ and so $ | x' | <n$; furthermore, it is not hard to see that $z=x'\bigtriangleup y'$. This gives us two elements $x',y'$ with $c(x')=c(y')=c(x'\bigtriangleup y')$ and $ | x' | < | x | $, which contradicts the choice of $x$ and $y$.

This contradiction shows that for no $x,y$ can the set $\fs(\{x,y\})$ be monochromatic. Hence the colouring $c$ witnesses $\B(2^\kappa)\nrightarrow(2)_\kappa^\fs$, and we are done.
  \end{proof}

The anonymous referee has pointed out that an unpublished result of Komj\'ath's extends Theorem~\ref{lowerbound} to all small abelian groups. That is, for every abelian group $G$ with $|G|\leq 2^\kappa$, it is the case that $G\nrightarrow(2)_\kappa^\fs$.

We finish this section with a brief discussion of the upper bounds for Komj\'ath's result on Boolean groups. Given an infinite cardinal $\kappa$ and an $n\in\mathbb N$, denote by $\beta(n,\kappa)=\min\{\lambda \mid \mathbb B(\lambda)\rightarrow(n)_\kappa^\fs\}$. Thus ~\cite[Theorem 1]{komjath} implies that $\beta(n,\kappa)\leq\beth_{2^{n-1}}(\kappa)^+$ (and in particular, $\sup_{n<\omega}\beta(n,\kappa)\leq\beth_\omega(\kappa)$). And our Theorem~\ref{lowerbound} states that $\beta(2,\kappa)=(2^\kappa)^+$ (and in particular, $2^\kappa<\sup_{n<\omega}\beta(n,\kappa)$). This yields the following characterization of strong limit cardinals in terms of the arrow relations satisfied by Boolean groups. The proof is immediate.

\begin{proposition}
Let $\lambda$ be an uncountable cardinal. Then the following are equivalent:
\begin{itemize}
\item $(\forall\kappa<\lambda)(\forall n\in\mathbb N)(\mathbb B(\lambda)\rightarrow(n)_\kappa^\fs)$;
\item $\lambda$ is a strong limit.
\end{itemize}
\end{proposition}

\section{Some negative results}\label{laseccioncorrespondiente}

In this section, we state and prove a number of negative Ramsey-theoretic results. Two of these arose from an (unsuccessful) attempt at generalizing analogous results from the context of finite colourings, when trying to increase the number of colours to an infinite number. There are two more results that also arise from an (also unsuccessful) attempt at generalizing some results that hold for Boolean groups.

The first result that we mention arises directly from Hindman's theorem, by considering the situation with infinite colourings. Recall that we denote by $\mathbb P$ the set of prime numbers along with $0$, and $\mathbb Z[0^\infty]$ is understood to be $\mathbb Q$ by convention.

\begin{theorem}\label{negativeinfinitecolours}
Let $G$ be an infinite abelian group. Then $G \nrightarrow(\omega)_\omega^{\fs}$. 
\end{theorem}

\begin{proof}
Suppose that $|G|=\lambda$. By the injectivity of divisible groups (as explained in the Introduction), $G$ embeds in $\bigoplus_{p\in\mathbb P}\left(\bigoplus_{\alpha\in I_p}Z[p^\infty]\right)$. We define the colouring $c$ on $G$ by $c(x):=(\sigma(\pi_p(x)) \mid p\in\mathbb P)$ (recall that $\sigma$ maps each element to its sequence of nonzero entries). This is a colouring with countably many colours, now we will show that there cannot be an infinite set $X$ with $\fs(X)$ monochromatic for this colouring. Suppose otherwise, and let $X$ be such that $c[\fs(X)]=\{(s_p \mid p\in\mathbb P)\}$. First of all, note that this implies, in particular, that all elements of $X$ have supports of the same size. Note, however, that it is not possible for two distinct $x,y\in X$ to have the exact same support, for, if $\supp(x)=\supp(y)$ and $(\sigma(\pi_p(x))\mid p\in\mathbb P)=c(x)=(\sigma_p\mid p\in\mathbb P)=c(y)=(\sigma(\pi_p(y))\mid p\in\mathbb P)$, then it must actually be the case that $x=y$. Hence the set $\{\supp(x)\mid x\in X\}$ is an infinite set, all of whose elements have the same fixed cardinality, and so we can apply the $\Delta$-system lemma\footnote{Or rather, the following version of the $\Delta$-system lemma: if $n<\omega$ and $X$ is an infinite family of sets of cardinality $n$, then there is an infinite $Y\subseteq X$ which forms a $\Delta$-system. This is easy to prove by induction on $n$, essentially using the technique outlined in~\cite[Exercise II.1, p. 86]{kunen}.} to it and obtain an infinite subset $Y\subseteq X$ that forms a $\Delta$-system. Now if we take $n+1$ distinct elements $y_0,\ldots,y_n\in Y$, the support of the sum $y_0+\cdots+y_n\in\fs(Y)\subseteq\fs(X)$ is guaranteed to have at least $n+1$ elements, contradicting that $c(y_0+\cdots+y_n)=n$. The proof is complete.
\end{proof}

We are grateful to the anonymous referee for pointing us toward using the $\Delta$-system lemma for the proof of Theorem~\ref{negativeinfinitecolours}, in a way that is reminiscent of~\cite[Theorem 1]{partitionpowerset}. The referee also directed us towards the following observation. By performing a standard compactness argument to the version of the $\Delta$-system that was used in the previous proof, one can obtain the following result: given any two fixed $k$ and $n$, there exists a number $\Delta(k,n)$ satisfying that whenever one has a sequence $x_1,\ldots,x_{\Delta(k,n)}$ of distinct sets of cardinality $k$, it is possible to take $n$ of them that form a $\Delta$-system. Thus, given a group $G$, if we let $c$ be the colouring of the proof of Theorem~\ref{negativeinfinitecolours}, and $\langle s_n\mid n<\omega\rangle$ be an injective sequence of $\ran(c)$, the same argument used in that proof actually shows that 
\begin{equation*}
G\nrightarrow(\Delta(|s_0|,|s_0|+1),\Delta(|s_1|,|s_1|+1),\ldots,\Delta(|s_n|,|s_n|+1),\ldots)^\fs,
\end{equation*}
meaning that there exists a colouring $c:G\longrightarrow\omega$ such that there is no $X$ of cardinality $\Delta(|s_0|,|s_0|+1)$ with $\fs(X)$ monochromatic in colour $0$, and there is no $X$ of cardinality $\Delta(|s_1|,|s_1|+1)$ with $\fs(X)$ monochromatic in colour $1$, and so on and so forth (since for any $X$ of size $\Delta(|s_n|,|s_n|+1)$ with $\fs(X)$ monochromatic in colour $s_n$, we would be able to pick $x_0,\ldots,x_{|s_n|}\in X$ forming a $\Delta$-system, and so the support of $x_0+\cdots+x_{|s_n|}$ would have at least $|s_n|+1$-many elements, a contradiction). In fact, by appropriately refining the colouring $c$ one can get a similar negative arrow relation where the sequence inside the parentheses is any arbitrary unbounded sequence. In particular, it is the case that every infinite abelian group $G$ satisfies the relation $G\nrightarrow(2,3,4,\ldots)^\fs$.

The following theorem arises from consideration of van der Waerden's theorem. Recall that van der Waerden's theorem~\cite{vanderwaerden} states that for every finite colouring of $\mathbb N$, there will be arbitrarily long arithmetic progressions that are monochromatic. In an abelian group $G$, we define an arithmetic progression of length $l$ to be a set of the form $\{a,a+b,\ldots,a+(l-1)b\}$ for $a,b\in G$ and $b\neq 0$. It is possible to generalize van der Waerden's theorem by proving that, in any infinite abelian group, every finite colouring of the group will give rise to arbitrarily long arithmetic progressions that are monochromatic (for example, such a statement follows from the Central Sets Theorem for abelian groups in the same way that the classical van der Waerden's theorem follows from the Central Sets Theorem for $\mathbb N$, see e.g. \cite[Corollary 14.13]{hindmanstrauss}). Now for infinite colourings, the situation splits: in a Boolean group, an arithmetic progression $\{a,a+b,\ldots,a+(l-1)b\}$ of length $l$ is simply equal to $\{a,a+b\}$ (as long as $l\geq2$), so every colouring of a Boolean group with less colours than the cardinality of the group is bound to have arbitrarily long arithmetic progressions, just by the pigeonhole principle. On the other hand, when we consider arbitrary abelian groups, with the restriction that they do not contain any elements of order $2$ (to completely discard the Boolean case), we obtain a negative result, even for the shortest possible length of an arithmetic progression that would render the statement nontrivial. We must point out that the following theorem, in the particular case where $G$ is the additive group of a vector space, was proved by Rado (see e.g. \cite[Theorem 3.2]{komjath-euclidean}). Although the proof is essentially the same, we still reproduce it here for the convenience of the reader. We are grateful to the anonymous referee for pointing out Rado's result, as well as reference~\cite{komjath-euclidean}.

\begin{theorem}\label{negativevanderwaerden}
    Let $G$ be an infinite abelian group with no elements of order 2. Then there exists $c: G\longrightarrow\omega$ such that for no two $a, b \in G$ with $b \ne 0$ do we have that $\{a, a+b, a+2b\}$ is monochromatic for $c$.
\end{theorem}

\begin{proof}
Embed $G$ in $\bigoplus_{p\in\mathbb P}\left(\bigoplus_{\alpha\in I_p}Z[p^\infty]\right)$. Just like in the proof of Theorem~\ref{negativeinfinitecolours}, we define the colouring $c$ on $G$ by letting $c(x)=(\sigma(\pi_p(x)) \mid p\in\mathbb P)$. Suppose that $a,b\in G$, with $b\neq 0$ are such that $\{a,a+b,a+2b\}$ is monochromatic, say with colour $(s_p \mid p\in\mathbb P)$. Since at least one of (in fact, at least two of) $a,a+b,a+2b$ is nonzero, and $G$ does not have elements of order $2$, it follows that $s_p\neq\varnothing$ for some $p\in\mathbb P\setminus\{2\}$. Let $\beta=\min(\supp(\pi_p(b)))$. Note that $b(\beta)\neq 2b(\beta)$, and both of these are nonzero. Thus if $a(\beta)\neq0$, we have that at least two of $a(\beta),a(\beta)+b(\beta),a(\beta)+2b(\beta)$ are nonzero, and for these two elements, this nonzero entry occupies the same index within $s_p$, however no two of these three elements are equal, a contradiction. On the other hand, if $a(\beta)=0$, then the entries $a(\beta)+b(\beta),a(\beta)+2b(\beta)$, which are distinct, correspond to the same element of $s_p$ for $a+b$ and $a+2b$, but since these entries are distinct, we get another contradiction.
\end{proof}

We now switch from considering arithmetic progressions, to considering subgroups. Given an abelian group $G$, we will use the symbol $G\rightarrow(\kappa)_\theta^{\langle\cdot\rangle}$ to denote the statement that for every colouring $c:G\longrightarrow\theta$, there exists a subgroup $H\subseteq G$ that can be generated with $\kappa$-many elements and such that $H\setminus\{0\}$ is monochromatic (by colouring $0$ with its own colour, different from everybody else's, we can ensure that no nontrivial group is monochromatic; hence we explicitly remove $0$ from our subgroups to avoid trivially false statements). If we look at an infinite Boolean group $\mathbb B$, by Hindman's theorem for every colouring of $\mathbb B$ there will be an infinite $X$ such that $\fs(X)$ is monochromatic. Note that, in the Boolean case, $\fs(X)$ coincides with $H\setminus\{0\}$, where\footnote{Provided that $X$ is a linearly independent set over $\mathbb Z/2\mathbb Z$. Otherwise, $\fs(X)=H$.} $H$ is the subgroup of $\mathbb B$ generated by $X$. Hence any infinite Boolean group $\mathbb B$ satisfies the statement $\mathbb B\rightarrow(\omega)_2^{\langle\cdot\rangle}$ (and by the Folkman--Rado--Sanders theorem, for every (finite) $n$ there exists an (finite) $m$ such that, if $\mathbb B$ is the Boolean group of size $2^m$, then $\mathbb B\rightarrow(n)_2^{\langle\cdot\rangle}$). The theorem below establishes that this property fails for any groups that do not have Boolean subgroups.

\begin{theorem}
	Let $G$ be an abelian group with no elements of order 2. Then $G \nrightarrow (1)^{\langle \cdot \rangle}_2$ (and \textit{a fortiori}, $G\nrightarrow(\kappa)_2^{\langle\cdot\rangle}$ for every nonzero cardinal $\kappa$).
\end{theorem}

\begin{proof}
Embed $G$ in $\bigoplus_{p\in\mathbb P}\left(\bigoplus_{\alpha\in I_p}Z[p^\infty]\right)$. We define the colouring $c$ on $G$ as follows. Given a nonzero (the colour of $0$ can be arbitrary) element $x\in G$, let $p_x=\min\{p\in\mathbb P\setminus\{2\} \mid \pi_p(x)\neq 0\}$, and let $\alpha_x=\min(\supp(\pi_{p_x}(x)))$. Let $q_x=x(\alpha_x)$, which will be either a rational number (if $p_x=0$), or the coset modulo $\mathbb Z$ of a rational number whose denominator is a power of $p_x$ (in which case we identify $q_x$ with the unique member of this coset that lies in $(0,1)$). We let $c(x)$ be $0$ or $1$ according to whether $\ord_2(q_x)$ is even or odd\footnote{Recall that, for $q\in\mathbb Q$, $\ord_2(q)=i$ if and only if $q=2^i\frac{a}{b}$, where $a,b\in\mathbb Z$ are coprime to $2$.}.

Suppose, for the sake of contradiction, that $H \subseteq G$ is a nontrivial subgroup with $H\setminus\{0\}$ monochromatic. Pick $x\in H\setminus\{0\}$. There are two cases, the first of which is if $p_x=0$. Then $q_x\in\mathbb Q$, and so clearly $\ord_2(q_{2x})=\ord_2(2q_x)=\ord_2(q_x)+1$, which implies that $x$ and $2x$ have distinct colours while both belonging to $H\setminus\{0\}$, a contradiction.

The second case is when $p_x$ is a positive prime number. In this case, $q_x$ must be of the form $\frac{a}{p_x^k}$ for some $k>0$, where $(a,p_x)=1$. Since $(a,p_x)=1$, the number $a$ must be a generator of the multiplicative group of units of the ring $\mathbb Z/p_x^k\mathbb Z$, which means that, if we pick our favourite $b$ satisfying $(b,p_x)=1$ and $0<b<\left\lfloor\frac{p_x^k}{2}\right\rfloor$, there will be an integer $c$ such that $ac\equiv b\mod p_x^k$. Thus if we let $y=cx$, we will have that $q_y=\frac{ca}{p_x^k}=\frac{b}{p_x^k}$, and $q_{2y}=2q_y=\frac{2b}{p_x^k}$, since $0<2b<p_x^k$ (that is, since $2b<p_x^k$, the representative in $(0,1)$ for the coset of $\frac{2b}{p_x^k}$ is itself). Now, clearly $\ord_2(q_{2y})=\ord_2\left(\frac{2b}{p_x^k}\right)=\ord_2(2b)=\ord_2(b)+1=\ord_2\left(\frac{b}{p_x^k}\right)+1=\ord_2(q_y)+1$, which means that $y$ and $2y$ have different colours, while both of them belong to $H\setminus\{0\}$, a contradiction.
\end{proof}

We will now address modules over rings. Let us consider first the field with two elements, $\mathbb F_2$. Every module over $\mathbb F_2$ is a Boolean group, and conversely every Boolean group is a module over $\mathbb F_2$. Furthermore, if $\mathbb B$ is a Boolean group and $X\subseteq\mathbb B$, then $\fs(X)$ is exactly\footnote{Unless $X$ is linearly dependent over $\mathbb F_2$, in which case $\fs(X)=H$.} $H\setminus\{0\}$, where $H$ is the span of the set $X$, when viewing $\mathbb B$ as a module over $\mathbb F_2$. Hence, if we let $M\rightarrow(\kappa)_\theta^\Span$ denote the statement that for every colouring of $M$ with $\theta$ colours there exists a submodule $N\subseteq M$ of dimension $\kappa$ such that $N\setminus\{0\}$ is monochromatic, we will have (as a consequence of Hindman's theorem) that $\mathbb B\rightarrow(\omega)_2^\Span$ whenever $\mathbb B$ is an infinite Boolean group (equivalently, an infinite $\mathbb F_2$-module). The theorem below establishes that we cannot have such a result for modules over other rings, at least if the modules are free.

\begin{theorem}\label{negativesubmodule}
	Let $R \ne \F_2$ be a Unique Factorization Domain (UFD)\footnote{That is, $R$ is a commutative ring with unit, without zero divisors, where every element can be decomposed as a product of irreducible elements in a unique way (up to order and associates).}. Then for every free $R$-module $M$, we have that $M \nrightarrow (1)_2^{\Span}$ (and \textit{a fortiori}, $M\nrightarrow(\kappa)_2^\Span$ for every positive $\kappa$).
\end{theorem}
\begin{proof}
   Since $R$ has no zero divisors, for each $a \in R\setminus\{0\}$ we have that $ar = as$ implies $r = s$. 
   Hence, the multiplicative action $f_a:R\longrightarrow R$ given by $f_a(r)=ar$ is an injection for every $a \in R\setminus\{0\}$. 
   Having chosen an arbitrary $a$, for every $r \in R$ let $S_r = \{a^ir:i \in \mathbb N\}$. 
   We claim that if $r, r' \in R$ are such that $S_r \cap S_{r'} \ne \varnothing$, then either $S_r \subset S_{r'}$ or $S_{r'} \subset S_r$. 
  For if $S_r \cap S_{r'} \ne \varnothing$, then by picking $a^nr = a^{n'}r'\in S_r\cap S_{r'}$, and assuming without loss of generality that $n \ge n'$, we see that $a^{n-n'}r = r'$, which implies that $r' \in S_r$ and therefore $S_{r'} \subset S_r$.
   
   Let us choose an irreducible element $a\in R$ for the above considerations. Then any chain of $S_{r_\alpha}$ must have a maximum.
   Suppose for the sake of contradiction that $(S_{r_\alpha}:\alpha < \lambda)$ is a strictly increasing sequence. 
   We may choose $r_\alpha$ such that for every $\xi < \alpha$,
   $r_\alpha \notin S_{r_\xi}$.
   With the same notation, we observe that $r_\xi = a^i r_\alpha$, and hence it follows that arbitrarily large powers of $a$ divide $r_0$.
   This is absurd in a UFD, which gives us our contradiction.
   
      By the axiom of choice, let $X \subset R$ be such that for every $r \in X$, $S_r$ is inclusion-maximal in the set $\{S_s \mid s \in R\}$ and for $r, r' \in X$ distinct, $S_r \cap S_{r'} = \varnothing$, and $\bigcup_{r \in X} S_r = R - \{0\}$. Then define $g: R \setminus\{0\} \to 2$ as $g(r) = i \mod 2$, where $r = a^ir_0$ for some $r_0 \in X$.
   
   If $\dim M = \lambda$ (where $\lambda$ need not be infinite), then set the colouring of $M = \bigoplus_\lambda R$ as follows:
   \[c(x) = g(x(\min(\supp(x))))\]
   Suppose that there is a submodule $N\subseteq M$ such that $N\setminus\{0\}$ is monochromatic, and let $x\in N\setminus\{0\}$. If $\alpha$ is the least index such that $x(\alpha) \ne 0$, notice that $g(ax(\alpha)) = g(x(\alpha)) + 1 \mod 2$. Thus, it follows that $c(x) =g(x(\alpha)) \ne g(ax(\alpha))=g((ax)(\alpha)) = c(ax)$, while $x,ax\in M$, a contradiction. 
\end{proof}

In~\cite{graham-rothschild}, it is proved that every finite colouring of an infinite vector space over a finite field yields arbitrarily large, but finite, monochromatic affine subspaces. If the field is $\mathbb F_2$, then in fact we can get an infinite monochromatic affine subspace (this statement follows directly from Hindman's theorem applied to Boolean groups). In~\cite{zelenyuk-affine} it is proved that every vector space over a finite field which is not $\mathbb F_2$ carries a colouring with countably many colours such that every infinite affine space meets all of the colours. The following result, along the lines of the ones just mentioned, is an immediate corollary of Theorem~\ref{negativevanderwaerden}. The one after that is another corollary of the same theorem, addressing the situation when one replaces ``affine subspace'' with ``translate of a subgroup''.

\begin{corollary}
   Let $R$ be a principal entire ring, $R\neq\mathbb F_2$, and $M$ be a free $R$-module. Then there is a colouring of $M$ with countably many colours such that no nontrivial affine subspace can possibly be monochromatic.
   \end{corollary}

\begin{corollary}
   Let $G$ be an abelian group without elements of order $2$. Then there exists $c: G \to \omega$ such that for no nontrivial subgroup $H \subseteq G$ and $a \in G$ can the set $a + H$ be monochromatic.
\end{corollary}

Both corollaries are proved by noting that, if $N\subseteq M$ is a submodule (resp. if $H\subseteq G$ is a subgroup) and $a\in M$ (resp. $a\in G$) then the affine subspace $a+N$ (resp. the set $a+H$) must contain $a,a+b,a+2b$ for some nonzero $b$.

\section*{Acknowledgements}
The first author would like to thank Peter Komj\'ath for sharing a preliminary version of~\cite{komjath}, which was the original motivation for this research; it is also worth mentioning that this author was partially supported by postdoctoral fellowship number 275049 from Conacyt--Mexico. The second author acknowledges partial support by NSF Grant \#DMS-1401384, as part of the University of Michigan Department of Mathematics REU program. Both authors are extremely grateful to the anonymous referee, for her careful reading of this paper and her multiple useful observations that helped improve it.

\end{document}